\documentclass[11pt]{article}

\usepackage{amsmath,amsthm,amssymb,amsfonts, graphicx}
\usepackage{graphics}
\usepackage{authblk}
\usepackage{upgreek}
\usepackage{amsrefs,hyperref}
\theoremstyle{plain}
\newtheorem{theorem}{Theorem}[section]

\newtheorem{corollary}[theorem]{Corollary}
\newtheorem{definition}[theorem]{Definition}

\newtheorem{lemma}[theorem]{Lemma}
\newtheorem{proposition}[theorem]{Proposition}
\newtheorem{remark}[theorem]{Remark}
\newtheorem*{theorem*}{Theorem}
\numberwithin{equation}{section}
\newcommand{\diff}{\mathop{}\!\mathrm{d}}

\DeclareMathOperator\supp{supp}

\title{Global regularity for solutions of the Navier--Stokes equation sufficiently close to being eigenfunctions of the Laplacian}
\author[1]{Evan Miller}
\affil[1]{McMaster University, Department of Mathematics

millee14@mcmaster.ca}

\begin{document}

\maketitle

\begin{abstract}
In this paper, we will prove a new, scale critical regularity criterion for solutions of the Navier--Stokes equation that are sufficiently close to being eigenfunctions of the Laplacian.
This estimate improves previous regularity criteria requiring control on the $\dot{H}^\alpha$ norm of $u,$ with $2\leq \alpha<\frac{5}{2},$ to a regularity criterion requiring control on the $\dot{H}^\alpha$ norm multiplied by the deficit in the interpolation inequality for the embedding of 
$\dot{H}^{\alpha-2}\cap\dot{H}^{\alpha} 
\hookrightarrow \dot{H}^{\alpha-1}.$
This regularity criterion suggests, at least heuristically, the possibility of some relationship between potential blowup solutions of the Navier--Stokes equation and the Kolmogorov-Obhukov spectrum in the theory of turbulence.
\end{abstract}

\section{Introduction}

The Navier--Stokes equation is one of the fundamental equations of fluid mechanics. For an incompressible fluid, where the density of the fluid in question is constant, the Navier--Stokes equation with no external forces is given by 
\begin{align}
    \partial_t u-\Delta u + P_{df}\left((u\cdot\nabla) u\right)&=0, \\
    \nabla \cdot u&=0,
\end{align}
where $u\in\mathbb{R}^3$ is the velocity of the fluid and $P_{df}$ is the projection onto the space of divergence free vector fields.
We have also taken the viscosity to be $1,$ which we can do without loss of generality, because it is equivalent up to rescaling. 
In his foundational paper on the subject, Leray proved the global existence of weak solutions to the Navier--Stokes equation satisfying an energy inequality \cite{Leray}; however, such solutions are not known to be either smooth or unique. 
This is because the bounds from the energy inequality
are both supercritical with respect to scaling.
Smooth solutions of the Navier--Stokes equation with initial data in $H^1$ must satisfy an energy equality, 
stating that for all $t>0,$
\begin{equation} \label{EnergyEquality}
    \frac{1}{2}\|u(\cdot,t)\|_{L^2}^2
    +\int_0^t \|\nabla u(\cdot,\tau)\|_{L^2}^2 \diff\tau
    =\frac{1}{2}\left\|u^0\right\|_{L^2}^2.
\end{equation}
The Navier--Stokes equation is invariant under the rescaling,
\begin{equation} \label{rescaling}
    u^\lambda(x,t)=\lambda u\left(\lambda x, \lambda^2 t\right),
\end{equation}
for all $\lambda>0$, and it is a simple calculation to check that the bounds on $u$ in $L^\infty_t L^2_x$ and $L^2_t \dot{H}^1_x$ from the energy equality \eqref{EnergyEquality} are supercritical with respect to the rescaling \eqref{rescaling}.

The only results presently available guaranteeing regularity for the Navier--Stokes equation with general, arbitrarily large initial data,
require control of some scale critical quantity. 
Ladyzhenskaya \cite{Ladyzhenskaya}, Prodi \cite{Prodi}, and Serrin \cite{Serrin}, proved that if a smooth solution of the Navier--Stokes equation blows up in finite time $T_{max}<+\infty,$ then
for all $3<q\leq+\infty, \frac{2}{p}+\frac{3}{q}=1,$
\begin{equation}
    \int_0^{T_{max}}\|u(\cdot,t)\|_{L^q}^p \diff t=+\infty.
\end{equation}
It is straightforward to check that $L^p_t L^q_x$ is scale critical with respect to the rescaling \eqref{rescaling}, when $\frac{2}{p}+\frac{3}{q}=1.$
This result was then extended by Escauriaza, Seregin, and \v{S}ver\'ak \cite{SereginL3} to the endpoint case $p=+ \infty, q=3.$ They proved that if a smooth solution $u$ of the Navier-Stokes equation blows up in finite time $T_{max}<+\infty$, then 
\begin{equation}
\limsup_{t \to T_{max}}\|u(\cdot,t)\|_{L^3(\mathbb{R}^3)}=+\infty.
\end{equation}

There have been a number of generalizations of the Ladyzhenskaya-Prodi-Serrin regularity criterion, including scale critical component reduction results involving the vorticity \cite{BealeKatoMajda}, two components of the vorticity \cite{ChaeVort}, the derivative in just one direction, $\frac{\partial u}{\partial x_i}$ \cite{Kukavica}, and involving only one component, $u_j$ \cites{Chemin1,Chemin2}.
The Ladyzhenskaya-Prodi-Serrin regularity criterion has also been generalized to endpoint Besov spaces
\cites{ChenZhangBesov,KOTbesov,KozonoShimadaBesov},
while the Escauriaza-Seregin-\v{S}ver\'ak regularity criteriona has been generalized to all non-endpoint Besov spaces \cites{AlbrittonBesov,GKPbesov,AlbrittonBarker}.

In this paper, we will generalize the Ladyzhenskaya-Prodi-Serrin regularity criterion to solutions of the Navier--Stokes equation that are close to being eigenfunctions of the Laplacian. One tool we will use is the notion of mild solutions, which was developed by Kato and Fujita \cite{Kato}.
\begin{definition}
\label{MildSolutions}
Suppose $u \in C\left([0,T);\dot{H}^1 
\left(\mathbb{R}^3\right) \right ),
\nabla\cdot u=0.$ 
Then $u$ is a mild solution to the Navier--Stokes equation if for all $0<t<T$,
\begin{equation}
u(\cdot,t)=e^{t \Delta}u^0
+\int_0^t e^{(t-\tau)\Delta}
P_{df}\left(-(u \cdot \nabla)u\right)
(\cdot, \tau)\diff \tau,
\end{equation}
where $e^{t\Delta}$ is the heat semi-group operator given by convolution with the heat kernel; that is to say, $e^{t\Delta}u^0$ is the solution of the heat equation after time $t,$ with initial data $u^0.$ 
\end{definition}
Fujita and Kato proved the local existence in time of mild solutions to the Navier--Stokes equation for all initial data in $\dot{H}^1\left(\mathbb{R}^3\right).$ Our results for solutions of the Navier--Stokes equation sufficiently close to being eigenfunctions of the Laplacian will be proven in terms of $H^1$ mild solutions.
We will also define, for all $\alpha>-\frac{3}{2},$ the homogeneous Sobolev space $\dot{H}^\alpha\left(\mathbb{R}^3\right)$,
which is a Hilbert space for $-\frac{3}{2}<\alpha<\frac{3}{2}$,
as the space with the norm
\begin{align}
    \|f\|_{\dot{H}^\alpha}^2
    &=
    \left\|(-\Delta)^{\frac{\alpha}{2}}f\right\|_{L^2}^2\\
    &=
    \int_{\mathbb{R}^3} \left( 2 \pi |\xi|\right)^{2\alpha}
    \big|\hat{f}(\xi)\big|^2 \diff\xi,
\end{align}
and the inhomogeneous Hilbert space
$H^\alpha\left(\mathbb{R}^3\right)$ 
as the space with the norm
\begin{equation}
    \|f\|_{H^\alpha}^2=
    \int_{\mathbb{R}^3} \left(1+4\pi^2|\xi|^2 \right)^{\alpha}
    \big|\hat{f}(\xi)\big|^2 \diff\xi,
\end{equation}
while further noting that for all $\alpha>0,$
\begin{equation}
    H^\alpha=L^2 \cap \dot{H}^\alpha.
\end{equation}
For solutions of the Navier--Stokes equation, we call 
$\frac{1}{2}\|u(\cdot,t)\|_{L^2}^2$ the energy and
$\frac{1}{2}\|u(\cdot,t)\|_{\dot{H}^1}^2$ the enstrophy.
With mild solutions and the relevant Sobolev spaces defined, we can now state the main theorem of this paper.

\begin{theorem} \label{MainTheoremIntro}
Suppose $u\in C\left(\left[0,T_{max}\right);H^1\right)$
is a mild solution of the Navier--Stokes equation, and suppose $\frac{6}{5}<q\leq 3, \frac{2}{p}+\frac{3}{q}=3.$ Then for all $0<t<T_{max}$
\begin{equation} \label{GrowthBound}
    \|\nabla u(\cdot,t)\|_{L^2}^2\leq 
    \left\|\nabla u^0\right\|_{L^2}^2
    \exp\left(C_q \int_0^t
    \inf_{\lambda(\tau)\in\mathbb{R}}
    \|-\Delta u-\lambda u\|_{L^q}^p \diff\tau \right),
\end{equation}
where $C_q>0$ depends only on $q.$
In particular, if $T_{max}<+\infty$ then
\begin{equation}
    \int_0^{T_{max}} \inf_{\lambda(t)\in\mathbb{R}}
    \|-\Delta u-\lambda u\|_{L^q}^p \diff t
    =+\infty.
\end{equation}
\end{theorem}

An eigenfunction of the Laplacian satisfies the equation,
\begin{equation}
    -\Delta u-\lambda u=0.
\end{equation}
There are no nonzero eigenfunctions of the Laplacian in
$H^1\left(\mathbb{R}^3\right),$ because if $-\Delta u=\lambda u,$ that would require its Fourier transform to be supported on a set of measure zero, specifically the set
\begin{equation}
    \left\{\xi\in\mathbb{R}^3: 4\pi^2|\xi|^2=\lambda\right\};
\end{equation} 
however, the quantity 
\begin{equation}
   \int_0^t \inf_{\lambda(\tau) \in\mathbb{R}}
\|-\Delta u-\lambda u\|_{L^q}^p \diff\tau
\end{equation} 
is nonetheless a scale critical measure of how close a solution is to being an eigenfunction of the Laplacian. This quantity is scale invariant because this infimum scales the same way as $-\Delta u,$ the quantity with no parameter in the infimum, and we can see from the scale invariance \eqref{rescaling}, that $-\Delta u$ has the scale invariance
\begin{equation}
    -\Delta u^\lambda(x,t)=
    -\lambda^3\Delta u \left(\lambda x, \lambda^2 t\right).
\end{equation}
It is a simple calculation to observe that when $\frac{2}{p}+\frac{3}{q}=3,$ the space $L^p_t L^q_x$ is invariant under this rescaling.

We do not have a nice expression for the quantity
\begin{equation}
    \inf_{\lambda \in\mathbb{R}}
    \|-\Delta u-\lambda u\|_{L^q},
\end{equation}
in general $L^q$ spaces, however in $L^2$ we can use the Hilbert space structure to calculate this quantity explicitly, which allows us to obtain the following result.

\begin{corollary} \label{Cor1Intro}
Suppose $u\in C\left(\left[0,T_{max}\right);H^1\right)$
is a mild solution of the Navier--Stokes equation. Then for all $0<t<T_{max}$
\begin{equation} 
    \|\nabla u(\cdot,t)\|_{L^2}^2\leq 
    \left\|\nabla u^0\right\|_{L^2}^2
    \exp\left(C_2 \int_0^t
    \|-\Delta u\|_{L^2}^\frac{4}{3}
    \left(1-\frac{\|\nabla u\|_{L^2}^4}
    {\|u\|_{L^2}^2 \|-\Delta u\|_{L^2}^2}
    \right)^\frac{2}{3} \diff\tau \right),
\end{equation}
where $C_2>0$ is taken as in Theorem \ref{MainTheoremIntro}.
In particular, if $T_{max}<+\infty$ then
\begin{equation}
    \int_0^{T_{max}} \|-\Delta u\|_{L^2}^\frac{4}{3}
    \left(1-\frac{\|\nabla u\|_{L^2}^4}
    {\|u\|_{L^2}^2 \|-\Delta u\|_{L^2}^2}
    \right)^\frac{2}{3} \diff t
    =+\infty.
\end{equation}
\end{corollary}

More generally, we can use the Sobolev embedding of $\dot{H}^\beta \hookrightarrow L^q$ to generalize Corollary \ref{Cor1Intro} in terms of homogeneous Sobolev spaces.

\begin{corollary} \label{Cor2Intro}
Suppose $u\in C\left(\left[0,T_{max}\right);H^1\right)$
is a mild solution of the Navier--Stokes equation, and suppose $2\leq \alpha\leq \frac{5}{2}, 
\alpha=\frac{1}{2}+\frac{2}{p}.$ 
Then for all $0<t<T_{max}$
\begin{equation} 
    \|\nabla u(\cdot,t)\|_{L^2}^2\leq 
    \left\|\nabla u^0\right\|_{L^2}^2
    \exp\left(\Tilde{C}_\alpha \int_0^t
    \|u\|_{\dot{H}^\alpha}^p 
    \left(1-\frac{\|u\|_{\dot{H}^{\alpha-1}}^4}
    {\|u\|_{\dot{H}^{\alpha-2}}^2 \|u\|_{\dot{H}^\alpha}^2}
    \right)^\frac{p}{2} \diff\tau \right),
\end{equation}
where $\Tilde{C}_\alpha>0$ depends only on $\alpha.$
In particular, if $T_{max}<+\infty$ then
\begin{equation}
    \int_0^{T_{max}}
    \|u\|_{\dot{H}^\alpha}^p 
    \left(1-\frac{\|u\|_{\dot{H}^{\alpha-1}}^4}
    {\|u\|_{\dot{H}^{\alpha-2}}^2 \|u\|_{\dot{H}^\alpha}^2}
    \right)^\frac{p}{2} \diff t
    =+\infty.
\end{equation}
Note that the scaling relation between $\alpha$ and $p$ can alternatively be expressed by
\begin{equation}
    p=\frac{2}{\alpha-\frac{1}{2}}
\end{equation}
\end{corollary}

We will note here that the $\alpha=2$ case of Corollary \ref{Cor2Intro}, is precisely Corollary \ref{Cor1Intro}. For $2<\alpha \leq \frac{5}{2},$ Corollary \ref{Cor2Intro} requires that we use the fractional Sobolev inequality to bound the infimum in $L^q$ by an infimum in the appropriate homogeneous Hilbert space, which can be calculated explicitly. We will show that
\begin{align}
    \inf_{\lambda \in\mathbb{R}}
    \|-\Delta u-\lambda u\|_{L^q}^2
    &\leq
    C \inf_{\lambda \in\mathbb{R}}
    \|-\Delta u-\lambda u\|_{\dot{H}^{\alpha-2}}^2\\
    &=
    C \|u\|_{\dot{H}^\alpha}^2 
    \left(1-\frac{\|u\|_{\dot{H}^{\alpha-1}}^4}
    {\|u\|_{\dot{H}^{\alpha-2}}^2 \|u\|_{\dot{H}^\alpha}^2}
    \right),
\end{align}
where $\alpha-2= \frac{3}{2}-\frac{3}{q}$.

We will also note that without the term $\left(1-\frac{\|u\|_{\dot{H}^{\alpha-1}}^4}
{\|u\|_{\dot{H}^{\alpha-1}}^2\|u\|_{\dot{H}^{\alpha-1}}^2}
\right)^\frac{p}{2},$  
Corollary \ref{Cor2Intro} is an immediate corollary of a variant of the Ladyzhenskaya-Prodi-Serrin regularity criterion.
Corollary \ref{Cor2Intro} shows that our regularity criterion for solutions of the Navier--Stokes equation sufficiently close to being eigenfunctions of the Laplacian measures the deficit in the interpolation inequality for the embedding 
\begin{equation}
    \dot{H}^{\alpha-2} \cap \dot{H}^{\alpha}
\hookrightarrow \dot{H}^{\alpha-1},
\end{equation}
which states that
\begin{equation} \label{HilbertInterpolate}
    \|f\|_{\dot{H}^{\alpha-1}}^2\leq
    \|f\|_{\dot{H}^\alpha}\|f\|_{\dot{H}^{\alpha-2}},
\end{equation}
where the constant $1$ is sharp but not attained, because there are no nonzero eigenfunctions of the Laplacian in 
$\dot{H}^{\alpha-1}\left(\mathbb{R}^3\right).$ 
When the inequality \eqref{HilbertInterpolate} is close to being saturated, then the quantity
\begin{equation}
    \left(1-\frac{\|u\|_{\dot{H}^{\alpha-1}}^4}
    {\|u\|_{\dot{H}^{\alpha-2}}^2 \|u\|_{\dot{H}^\alpha}^2}
    \right)^\frac{p}{2}
\end{equation}
will be small, so Corollary \ref{Cor2Intro} limits the extent to which blowup solutions can saturate the interpolation inequality \eqref{HilbertInterpolate}.

We will also prove that finite-time blowup solutions cannot concentrate on arbitrarily narrow bands in Fourier space, supported between an inner radius of $R_1(t)$ and an outer radius of $R_2(t),$ 
with the ratio $\frac{R_1(t)}{R_2(t)} \to 1$ 
arbitrarily quickly as $t \to T_{max}$ relative to the 
size of $\|u(\cdot,t)\|_{\dot{H}^\alpha}^p$.

\begin{corollary} \label{NarrowBandsRegCritIntro}
Suppose $u\in C\left(\left[0,T_{max}\right);H^1\right)$
is a mild solution of the Navier--Stokes equation,
and suppose for all $0<t<T_{max}$
\begin{equation}
    \supp{ \hat{u}(t)} \subset \left\{ \xi\in\mathbb{R}^3:
    R_1(t) \leq |\xi| \leq R_2(t) \right\}.
\end{equation}
Let $2\leq \alpha\leq \frac{5}{2}, 
\alpha=\frac{1}{2}+\frac{2}{p}.$ 
Then for all $0<t<T_{max}$
\begin{equation} 
    \|\nabla u(\cdot,t)\|_{L^2}^2\leq 
    \left\|\nabla u^0\right\|_{L^2}^2
    \exp\left(\Tilde{C}_\alpha \int_0^t
    \|u\|_{\dot{H}^\alpha}^p 
    \left(1-\frac{R_1(\tau)^4}{R_2(\tau))^4}
    \right)^\frac{p}{2} \diff\tau \right),
\end{equation}
where $\Tilde{C}_\alpha>0$ depends only on $\alpha.$
In particular, if $T_{max}<+\infty$ then
\begin{equation}
    \int_0^{T_{max}}
    \|u\|_{\dot{H}^\alpha}^p 
    \left(1-\frac{R_1(t)^4}{R_2(t))^4}
    \right)^\frac{p}{2} \diff t
    =+\infty.
\end{equation}
\end{corollary}

There are a number of previous results for Navier--Stokes regularity criteria involving frequency localization in Fourier space \cites{Bradshaw,Luo,Shvydkoy}. There is not enough space to discuss these results in detail here, as stating the main theorems would require us to define a number of objects from Littlewood-Paley theory, but we will note that the regularity criterion in \cite{Shvydkoy} has an explicit connection with the Kolmogorov scaling in turbulence.
Corollary \ref{NarrowBandsRegCritIntro} can also be seen as providing a heuristic connection between the regularity criterion in Theorem \ref{MainTheoremIntro}, and the Kolmogorov phenomenological theory of turbulence.
Corollary \ref{NarrowBandsRegCritIntro} shows that solutions supported on narrow bands in Fourier space are not good candidates for finite-time blowup; this is consistent with the Kolmogorov-Obukhov phenomenology of turbulence
\cites{Kolmogorov,Obukhov}, which stipulates that turbulent flows cannot localize around a small number of frequencies, specifically that 
the energy spectrum for turbulent flows has a decay in Fourier space on the order of
$|\xi|^{-\frac{5}{3}}$ in the inertial range.

\section{Proofs of the results}

Before beginning the proof of the main theorem, we will first need to establish an identity for the growth of enstrophy related to eigenfunctions of the Laplacian.

\begin{lemma} \label{GrowthID}
Suppose $u\in C\left(\left[0,T_{max}\right);H^1\right)$
is a mild solution of the Navier--Stokes equation. Then for all $0<t<T_{max},$ and for all $\lambda(t)\in \mathbb{R}$
\begin{equation}
    \frac{\diff}{\diff t} \frac{1}{2}\|\nabla u(\cdot,t)\|_{L^2}^2
    =-\|\Delta u\|_{L^2}^2 
    -\left<-\Delta u-\lambda u,(u\cdot\nabla)u\right>.
\end{equation}
\end{lemma}
\begin{proof}
It is easy to see from the Navier--Stokes equation that
\begin{equation}
    \frac{\diff}{\diff t} \frac{1}{2}\|\nabla u(\cdot,t)\|_{L^2}^2
    =
    -\|-\Delta u\|_{L^2}^2-\left<-\Delta u, (u\cdot\nabla)u\right>.
\end{equation}
We know that $u\in H^1,$ and therefore we have sufficient regularity to integrate by parts, using the condition $\nabla \cdot u=0$, to conclude that
\begin{align}
    \left<(u\cdot\nabla)u,u\right>
    &=
    -\left<u,(u\cdot\nabla)u\right>\\
    &=0.
\end{align}
Because we know that for all $\lambda\in\mathbb{R}$
\begin{equation}
    \left<\lambda u,(u\cdot\nabla)u\right>=0,
\end{equation}
we can conclude that 
for all $\lambda(t)\in \mathbb{R}$,
\begin{equation}
    \frac{\diff}{\diff t} \frac{1}{2}\|\nabla u(\cdot,t)\|_{L^2}^2
    =-\|\Delta u\|_{L^2}^2 
    -\left<-\Delta u-\lambda u,(u\cdot\nabla)u\right>.
\end{equation}
This completes the proof.
\end{proof}

One of the other main ingredients in our proof will be the fractional Sobolev inequality, which is stated below.
\begin{theorem} \label{Sobolev}
    Suppose $0<s<\frac{3}{2},$ and 
    $\frac{1}{q}=\frac{1}{2}-\frac{s}{3}.$
    Then for all 
    $f\in\dot{H}^s\left(\mathbb{R}^3\right),$
    \begin{equation}
        \|f\|_{L^q}\leq C_s \|f\|_{\dot{H}^s},
    \end{equation}
    where
    \begin{equation}
        C_s=2^{-\frac{s}{3}}\pi^{-\frac{4}{3}s}
        \left(\frac{\Gamma\left(\frac{3}{2}-s\right)}
        {\Gamma\left(\frac{3}{2}+s\right)}
        \right)^\frac{1}{2}
    \end{equation}
    Note that the scaling relation between the parameters $q$ and $s$ can be stated equivalently as 
    \begin{equation}
        s=\frac{3}{2}-\frac{3}{q}.
    \end{equation}
\end{theorem}
The Sobolev inequality was first proven by Sobolev \cite{Sobolev} in the case where $s=1$.
The sharp version of this inequality was proven by Talenti \cite{Talenti} in the case where $s=1$, and the general sharp version of this inequality with $0<s<\frac{3}{2}$ was proven by Lieb \cite{Lieb}.
With these results established, we can now prove Theorem \ref{MainTheoremIntro}, which is restated here for the reader's convenience.

\begin{theorem} \label{MainTheorem}
Suppose $u\in C\left(\left[0,T_{max}\right);H^1\right)$
is a mild solution of the Navier--Stokes equation, and suppose $\frac{6}{5}<q\leq 3, \frac{2}{p}+\frac{3}{q}=3.$ Then for all $0<t<T_{max}$
\begin{equation} \label{GrowthBoundBody}
    \|\nabla u(\cdot,t)\|_{L^2}^2\leq 
    \left\|\nabla u^0\right\|_{L^2}^2
    \exp\left(C_q \int_0^t
    \inf_{\lambda(\tau)\in\mathbb{R}}
    \|-\Delta u-\lambda u\|_{L^q}^p \diff\tau \right),
\end{equation}
where $C_q>0$ depends only on $q.$
In particular, if $T_{max}<+\infty$ then
\begin{equation}
    \int_0^{T_{max}} \inf_{\lambda(t)\in\mathbb{R}}
    \|-\Delta u-\lambda u\|_{L^q}^p \diff t
    =+\infty.
\end{equation}
\end{theorem}

\begin{proof}
We know that if $T_{max}<+\infty,$ then
\begin{equation}
    \lim_{t \to T_{max}}\|\nabla u(\cdot,t)\|_{L^2}=+\infty,
\end{equation}
so it suffices to prove the bound \eqref{GrowthBoundBody}. We will not keep track of the value of the constants $C_q.$ It is possible to compute the value of the constant $C_q$ explicitly in terms of the sharp fractional Sobolev inequality in Theorem \ref{Sobolev}, but we will not concern ourselves with long expressions for the value of the constant that would only clutter up this paper without adding any real mathematical insight.

First we will consider the case when $q=3, p=1.$ Using the identity for enstrophy growth in Proposition \ref{GrowthID}, and applying H\"older's inequality and the Sobolev inequality, we find that for all $0<t<T_{max},$ and for all $\lambda(t)\in\mathbb{R},$
\begin{align}
    \frac{\diff}{\diff t} \frac{1}{2}\|\nabla u(\cdot,t)\|_{L^2}^2
    &=
    -\|\Delta u\|_{L^2}^2 
    -\left<-\Delta u-\lambda u,(u\cdot\nabla)u\right>\\
    &\leq
    \|-\Delta u-\lambda u\|_{L^3} \|u\|_{L^6}\|\nabla u\|_{L^2}\\
    &\leq 
    C \|-\Delta u-\lambda u\|_{L^3}\|\nabla u\|_{L^2}^2.
\end{align}
Multiplying both sides by $2$ and taking the infimum over $\lambda(t) \in \mathbb{R},$ we find that
\begin{equation}
    \frac{\diff}{\diff t} \|\nabla u(\cdot,t)\|_{L^2}^2
    \leq C_3 \inf_{\lambda(t)\in\mathbb{R}}
    \|-\Delta u-\lambda u\|_{L^3}\|\nabla u\|_{L^2}^2.
\end{equation}
Applying Gr\"onwall's inequality we find that
for all $0<t<T_{max}$
\begin{equation}
    \|\nabla u(\cdot,t)\|_{L^2}^2\leq 
    \left\|\nabla u^0\right\|_{L^2}^2
    \exp\left(C_3 \int_0^t
    \inf_{\lambda(\tau)\in\mathbb{R}}
    \|-\Delta u-\lambda u\|_{L^3} \diff\tau \right),
\end{equation}
and we are done with the case where $q=3.$

Now we will consider the case $\frac{6}{5}<q<3.$
First we will take $6<a<+\infty$ to be given by
\begin{equation} \label{DefA}
    \frac{1}{a}=\frac{5}{18}-\frac{1}{3q},
\end{equation}
and $2<b<6$ to be given by 
\begin{equation} \label{DefB}
    \frac{1}{b}=\frac{13}{18}-\frac{2}{3q}.
\end{equation}
Note that 
\begin{equation}
    \frac{1}{a}+\frac{1}{b}+\frac{1}{q}=1,
\end{equation}
so again using the identity for enstrophy growth from Lemma \ref{GrowthID}, H\"older's inequality, and the fractional Sobolev inequality, we find that for all $0<t<T_{max}$ and for all $\lambda(t)\in\mathbb{R},$
\begin{align}
    \frac{\diff}{\diff t} 
    \frac{1}{2}\|\nabla u(\cdot,t)\|_{L^2}^2
    &=
    -\|\Delta u\|_{L^2}^2 
    -\left<-\Delta u-\lambda u,(u\cdot\nabla)u\right>\\
    &\leq 
    -\|\Delta u\|_{L^2}^2
    +\|\Delta u-\lambda u\|_{L^q}\|u\|_{L^a}\|\nabla u\|_{L^b}\\
    &\leq
    -\|\Delta u\|_{L^2}^2
    +C\|\Delta u-\lambda u\|_{L^q}
    \|\nabla u\|_{\dot{H}^\alpha}\|\nabla u\|_{\dot{H}^\beta},
\end{align}
where 
\begin{equation}
    \alpha=\frac{1}{2}-\frac{3}{a},
\end{equation}
and
\begin{equation}
    \beta=\frac{3}{2}-\frac{3}{b}.
\end{equation}
Plugging back into \eqref{DefA} and \eqref{DefB}, we find that
\begin{equation}
    \alpha=\frac{1}{q}-\frac{1}{3},
\end{equation}
and
\begin{equation}
    \beta=\frac{2}{q}-\frac{2}{3}.
\end{equation}
It is straightforward to see that 
$0<\alpha<\frac{1}{2}$ and $0<\beta<1,$ so we can interpolate between $L^2$ and $\dot{H}^1$ to find that
\begin{align}
    \frac{\diff}{\diff t}
    \frac{1}{2}\|\nabla u(\cdot,t)\|_{L^2}^2
    &\leq
    -\|\Delta u\|_{L^2}^2
    +C\|\Delta u-\lambda u\|_{L^q}
    \|\nabla u\|_{L^2}^{2-(\alpha+\beta)}
    \|\nabla u\|_{\dot{H}^1}^{\alpha+\beta}\\
    &=
    -\|\Delta u\|_{L^2}^2
    +C \|\Delta u-\lambda u\|_{L^q}
    \|\nabla u\|_{L^2}^{3-\frac{3}{q}}
    \|-\Delta u\|_{L^2}^{\frac{3}{q}-1},
\end{align}
where we have used the fact that $\alpha+\beta=\frac{3}{q}-1.$
Recalling that $\frac{2}{p}=3-\frac{3}{q},$ we can see that $\frac{6}{5}<q<3$ implies that $1<p<4.$ Take $\frac{4}{3}<r<+\infty$ to be the conjugate of p.
\begin{equation}
    \frac{1}{p}+\frac{1}{r}=1.
\end{equation}
Observe that
\begin{align}
    \frac{2}{r}
    &=2-\frac{2}{p}\\
    &=\frac{3}{q}-1.
\end{align}
Using this to simplify and applying Young's inequality with exponents $r$ and $p,$ we find that
\begin{align}
    \frac{\diff}{\diff t} 
    \frac{1}{2}\|\nabla u(\cdot,t)\|_{L^2}^2
    &\leq
    -\|\Delta u\|_{L^2}^2
    +C\|-\Delta u\|_{L^2}^{\frac{2}{r}}
    \|\Delta u-\lambda u\|_{L^q}
    \|\nabla u\|_{L^2}^{\frac{2}{p}}\\
    &\leq
    C\|\Delta u-\lambda u\|_{L^q}^p \|\nabla u\|_{L^2}^2
\end{align}
Multiplying both sides by $2$ and taking the infimum over $\lambda(t)\in\mathbb{R},$ we find that
\begin{equation}
     \frac{\diff}{\diff t}
     \|\nabla u(\cdot,t)\|_{L^2}^2
    \leq
    C_q \inf_{\lambda(t)\in\mathbb{R}}
    \|\Delta u-\lambda u\|_{L^q}^p \|\nabla u\|_{L^2}^2
\end{equation}
Applying Gr\"onwall's inequality we find that for all $0<t<T_{max}$
\begin{equation}
    \|\nabla u(\cdot,t)\|_{L^2}^2\leq 
    \left\|\nabla u^0\right\|_{L^2}^2
    \exp\left(C_q \int_0^t 
    \inf_{\lambda(\tau)\in\mathbb{R}}
    \|-\Delta u-\lambda u\|_{L^q}^p \diff\tau \right).
\end{equation}
This completes the proof.
\end{proof}

We will note here that the key element of the proof is the fact that
\begin{equation}
    \left<(u\cdot\nabla)u,u\right>=0.
\end{equation}
Because Tao's averaged 3D Navier--Stokes model equation also has the property
\begin{equation}
    \left<\Tilde{B}(u,u),u\right>=0,
\end{equation}
the regularity criterion in Theorem \ref{MainTheorem} and the subsequent corollaries will also apply to Tao's model equation, for which there is finite-time blowup \cite{TaoModel}. 

For general $\frac{6}{5}<q\leq 3$ we cannot compute
\begin{equation}
   \inf_{\lambda \in\mathbb{R}}
    \|-\Delta u-\lambda u\|_{L^q} 
\end{equation}
explicitly, but in the special case where $q=2,$ we can compute this infimum explicitly by making use of the Hilbert space structure.

\begin{proposition} \label{InfCalc}
For all $u\in H^2\left(\mathbb{R}^3\right),$ $u$ not identically zero,
\begin{equation}
    \inf_{\lambda \in\mathbb{R}}
    \|-\Delta u-\lambda u\|_{L^2}^2
    =
    \|-\Delta u\|_{L^2}^2
    \left(1-\frac{\|\nabla u\|_{L^2}^4}
    {\|u\|_{L^2}^2 \|-\Delta u\|_{L^2}^2}
    \right)
\end{equation}
\end{proposition}

\begin{proof}
Fix $u\in H^2.$ Define $f:\mathbb{R} \to \mathbb{R},$ by
\begin{align}
    f(\lambda)
    &=
    \|-\Delta u-\lambda u\|_{L^2}^2\\
    &=
    \|-\Delta u\|_{L^2}^2 -2\left<-\Delta u, u\right>\lambda
    +\|u\|_{L^2}^2 \lambda^2\\
    &=
    \|-\Delta u\|_{L^2}^2 -2\|\nabla u\|_{L^2}^2\lambda
    +\|u\|_{L^2}^2 \lambda^2
\end{align}
Taking the derivative of $f$ we find that
\begin{equation}
    f'(\lambda)=-2\|\nabla u\|_{L^2}^2+2\|u\|_{L^2}^2 \lambda.
\end{equation}
Let $\lambda_0=\frac{\|\nabla u\|_{L^2}^2}{\|u\|_{L^2}^2}.$
It is easy to see that for all $\lambda<\lambda_0, f'(\lambda)<0,$ for all $\lambda>\lambda_0, f'(\lambda)>0,$ and $f'(\lambda_0)=0,$
so we can conclude that $f$ has a global minimum at $\lambda_0.$
Therefore we can compute that
\begin{align}
    \inf_{\lambda \in\mathbb{R}}
    \|-\Delta u-\lambda u\|_{L^2}^2
    &=
    \inf_{\lambda \in\mathbb{R}} f(\lambda)\\
    &=
    f(\lambda_0)\\
    &=
    \|-\Delta u\|_{L^2}^2
    -\frac{\|\nabla u\|_{L^2}^4}{\|u\|_{L^2}^2} \\
    &=
    \|-\Delta u\|_{L^2}^2
    \left(1-\frac{\|\nabla u\|_{L^2}^4}
    {\|u\|_{L^2}^2 \|-\Delta u\|_{L^2}^2}\right).
\end{align}
This completes the proof.
\end{proof}

Using this identity for the infimum in the case where $q=2,$ we will now prove Corollary \ref{Cor1Intro}, which is restated here for the reader's convenience.

\begin{corollary} \label{Cor1}
Suppose $u\in C\left(\left[0,T_{max}\right);H^1\right)$
is a mild solution of the Navier--Stokes equation. Then for all $0<t<T_{max}$
\begin{equation} 
    \|\nabla u(\cdot,t)\|_{L^2}^2\leq 
    \left\|\nabla u^0\right\|_{L^2}^2
    \exp\left(C_2 \int_0^t
    \|-\Delta u\|_{L^2}^\frac{4}{3}
    \left(1-\frac{\|\nabla u\|_{L^2}^4}
    {\|u\|_{L^2}^2 \|-\Delta u\|_{L^2}^2}
    \right)^\frac{2}{3} \diff\tau \right),
\end{equation}
where $C_2>0$ is taken as in Theorem \ref{MainTheorem}.
In particular, if $T_{max}<+\infty$ then
\begin{equation}
    \int_0^{T_{max}} \|-\Delta u\|_{L^2}^\frac{4}{3}
    \left(1-\frac{\|\nabla u\|_{L^2}^4}
    {\|u\|_{L^2}^2 \|-\Delta u\|_{L^2}^2}
    \right)^\frac{2}{3} \diff t
    =+\infty.
\end{equation}
\end{corollary}

\begin{proof}
We will begin by observing that when $q=2, p=\frac{4}{3},$ then
\begin{equation}
    \frac{2}{p}+\frac{3}{q}=3.
\end{equation}
Next we know from Proposition \ref{InfCalc}, that
\begin{equation}
        \inf_{\lambda \in\mathbb{R}}
    \|-\Delta u-\lambda u\|_{L^2}^2
    =
    \|-\Delta u\|_{L^2}^2
    \left(1-\frac{\|\nabla u\|_{L^2}^4}
    {\|u\|_{L^2}^2 \|-\Delta u\|_{L^2}^2}
    \right).
\end{equation}
Taking both sides of the equation to the $\frac{2}{3}$ power, we find that
\begin{equation} \label{CorStep}
    \inf_{\lambda \in\mathbb{R}}
    \|-\Delta u-\lambda u\|_{L^2}^\frac{4}{3}
    =
    \|-\Delta u\|_{L^2}^\frac{4}{3}
    \left(1-\frac{\|\nabla u\|_{L^2}^4}
    {\|u\|_{L^2}^2 \|-\Delta u\|_{L^2}^2}
    \right)^\frac{2}{3},
\end{equation}
and then the result follows as an immediate corollary of Theorem \ref{MainTheorem}.
\end{proof}

Now, we will note that while we cannot explicitly compute the infimum in Theorem \ref{MainTheorem}, for $2<q<3,$ we can compute this infimum in the Hilbert space 
$\dot{H}^\beta \hookrightarrow L^q,$ using the inner product structure, and this will give us an explicit, scale-critical regularity criterion, albeit one requiring a higher degree of regularity.

\begin{proposition} \label{InfCalc2}
Suppose $0\leq \beta<\frac{3}{2}.$
For all $u\in \dot{H}^\beta\left(\mathbb{R}^3\right)
\cap \dot{H}^{2+\beta} \left(\mathbb{R}^3\right),$ 
$u$ not identically zero,
\begin{equation}
    \inf_{\lambda \in\mathbb{R}}
    \|-\Delta u-\lambda u\|_{\dot{H}^\beta}^2
    =
    \|u\|_{\dot{H}^{2+\beta}}^2
    \left(1-\frac{\|u\|_{\dot{H}^{1+\beta}}^4}
    {\|u\|_{\dot{H}^\beta}^2
    \|u\|_{\dot{H}^{2+\beta}}^2}\right).
\end{equation}
\end{proposition}

\begin{proof}
Fix $0\leq \beta<\frac{3}{2},$ and 
$u\in \dot{H}^\beta\left(\mathbb{R}^3\right)
\cap H^{2+\beta} \left(\mathbb{R}^3\right).$ 
Define $f:\mathbb{R} \to \mathbb{R},$ by
\begin{align}
    f(\lambda)
    &=
    \|-\Delta u-\lambda u\|_{\dot{H}^\beta}^2\\
    &=
    \|-\Delta u\|_{\dot{H}^\beta}^2 -2\left<-\Delta u, u\right>_{\dot{H}^\beta}\lambda
    +\|u\|_{\dot{H}^\beta}^2 \lambda^2\\
    &=
    \|u\|_{\dot{H}^{2+\beta}}^2
    -2\|u\|_{\dot{H}^{1+\beta}}^2\lambda
    +\|u\|_{\dot{H}^\beta}^2 \lambda^2
\end{align}
Taking the derivative of $f$ we find that
\begin{equation}
    f'(\lambda)=-2\|u\|_{\dot{H}^{1+\beta}}^2
    +2 \|u\|_{\dot{H}^\beta}^2 \lambda
\end{equation}
Let $\lambda_0=\frac{\|u\|_{\dot{H}^{1+\beta}}^2}
{\|u\|_{\dot{H}^\beta}^2}.$
It is easy to see that for all $\lambda<\lambda_0, f'(\lambda)<0,$ for all $\lambda>\lambda_0, f'(\lambda)>0,$ and $f'(\lambda_0)=0,$
so we can conclude that $f$ has a global minimum at $\lambda_0.$
Therefore we can compute that
\begin{align}
    \inf_{\lambda \in\mathbb{R}}
    \|-\Delta u-\lambda u\|_{\dot{H}^\beta}^2
    &=
    \inf_{\lambda \in\mathbb{R}} f(\lambda)\\
    &=
    f(\lambda_0)\\
    &=
    \|u\|_{\dot{H}^{2+\beta}}^2
    -\frac{\|u\|_{\dot{H}^{1+\beta}}^4}
    {\|u\|_{\dot{H}^\beta}^2} \\
    &=
    \|u\|_{\dot{H}^{2+\beta}}^2
    \left(1-\frac{\|u\|_{\dot{H}^{1+\beta}}^4}
    {\|u\|_{\dot{H}^\beta}^2
    \|u\|_{\dot{H}^{2+\beta}}^2}\right).
\end{align}
This completes the proof.
\end{proof}

Using Proposition \ref{InfCalc2}, and the Sobolev inequality corresponding to the embedding 
$\dot{H}^\beta \hookrightarrow L^q,$ we will now prove Corollary \ref{Cor2Intro}, which is restated here for the reader's convenience.

\begin{corollary} \label{Cor2}
Suppose $u\in C\left(\left[0,T_{max}\right);H^1\right)$
is a mild solution of the Navier--Stokes equation, and suppose $2\leq \alpha\leq \frac{5}{2}, 
\alpha=\frac{1}{2}+\frac{2}{p}.$ 
Then for all $0<t<T_{max}$
\begin{equation} 
    \|\nabla u(\cdot,t)\|_{L^2}^2\leq 
    \left\|\nabla u^0\right\|_{L^2}^2
    \exp\left(\Tilde{C}_\alpha \int_0^t
    \|u\|_{\dot{H}^\alpha}^p 
    \left(1-\frac{\|u\|_{\dot{H}^{\alpha-1}}^4}
    {\|u\|_{\dot{H}^{\alpha-2}}^2 \|u\|_{\dot{H}^\alpha}^2}
    \right)^\frac{p}{2} \diff\tau \right),
\end{equation}
where $\Tilde{C}_\alpha>0$ depends only on $\alpha.$
In particular, if $T_{max}<+\infty$ then
\begin{equation}
    \int_0^{T_{max}}
    \|u\|_{\dot{H}^\alpha}^p 
    \left(1-\frac{\|u\|_{\dot{H}^{\alpha-1}}^4}
    {\|u\|_{\dot{H}^{\alpha-2}}^2 \|u\|_{\dot{H}^\alpha}^2}
    \right)^\frac{p}{2} \diff t
    =+\infty.
\end{equation}
Note that the scaling relation between $\alpha$ and $p$ can alternatively be expressed by
\begin{equation}
    p=\frac{2}{\alpha-\frac{1}{2}}
\end{equation}
\end{corollary}

\begin{proof}
In the case where $\alpha=2,$ this is precisely the same statement as Corollary \ref{Cor1}, 
so fix $2<\alpha\leq \frac{5}{2}.$ 
Let $\beta=\alpha-2,$ so we have $0<\beta\leq \frac{1}{2}.$ 
Let $\frac{1}{q}=\frac{1}{2}-\frac{\beta}{3},$
so we have the Sobolev embedding
$\dot{H}^\beta\left(\mathbb{R}^3\right)
\hookrightarrow L^q\left(\mathbb{R}^3\right).$
Then we can see that
\begin{align}
\frac{3}{q}&=\frac{3}{2}-\beta\\
&= \frac{7}{2}-\alpha.
\end{align}
Likewise we know that
\begin{equation}
    \frac{2}{p}=\alpha-\frac{1}{2},
\end{equation}
so we can conclude that
\begin{equation}
    \frac{2}{p}+\frac{3}{q}=3.
\end{equation}
We know from the Sobolev inequality and from Proposition \ref{InfCalc2} that
\begin{align}
    \inf_{\lambda \in\mathbb{R}}
    \|-\Delta u-\lambda u\|_{L^q}^2
    &\leq 
    C \inf_{\lambda \in\mathbb{R}}
    \|-\Delta u-\lambda u\|_{\dot{H}^\beta}^2\\
    &=
    C \|u\|_{\dot{H}^{2+\beta}}^2
    \left(1-\frac{\|u\|_{H^{1+\beta}}^4}
    {\|u\|_{\dot{H}^\beta}^2 \|u\|_{\dot{H}^{2+\beta}}^2}
    \right)\\
    &=
    \|u\|_{\dot{H}^\alpha}^2 
    \left(1-\frac{\|u\|_{\dot{H}^{\alpha-1}}^4}
    {\|u\|_{\dot{H}^{\alpha-2}}^2 \|u\|_{\dot{H}^\alpha}^2}
    \right)
\end{align}
Taking both sides of the equation to the power of $\frac{p}{2}$ we find that
\begin{equation}
    \inf_{\lambda \in\mathbb{R}}
    \|-\Delta u-\lambda u\|_{L^q}^p
    \leq
    C \|u\|_{\dot{H}^\alpha}^p 
    \left(1-\frac{\|u\|_{\dot{H}^{\alpha-1}}^4}
    {\|u\|_{\dot{H}^{\alpha-2}}^2 \|u\|_{\dot{H}^\alpha}^2}
    \right)^\frac{p}{2},
\end{equation}
and we have already shown that
$\frac{2}{p}+\frac{3}{q}=3,$
so the result then follows as an immediate corollary of Theorem \ref{MainTheorem}.
\end{proof}

We mentioned in the introduction that without the term $\left(1-\frac{\|u\|_{\dot{H}^{\alpha-1}}^4}
{\|u\|_{\dot{H}^{\alpha-1}}^2\|u\|_{\dot{H}^{\alpha-1}}^2}
\right)^\frac{p}{2},$  
Corollary \ref{Cor2Intro} is an immediate corollary of the a variant of the Ladyzhenskaya-Prodi-Serrin regularity criterion, which states that for a smooth solution of the Navier--Stokes equation, if $T_{max}<+\infty$ and we have $\frac{2}{p}+\frac{3}{q}=2, \frac{3}{2}<q\leq +\infty,$ then 
\begin{equation}
    \int_0^{T_{max}} \|\nabla u\|_{L^q}^p \diff t=+\infty.
\end{equation}
This means that if $T_{max}<+\infty,$ then for all $1 \leq \alpha< \frac{5}{2}, \alpha=\frac{1}{2}+\frac{2}{p},$
\begin{align}
    \int_0^{T_{max}} \|u\|_{\dot{H}^\alpha}^p \diff t
    &=
    \int_0^{T_{max}} \|\nabla u\|_{\dot{H}^{\alpha-1}}^p \diff t\\
    &\geq
    C \int_0^{T_{max}} \|\nabla u\|_{L^q}^p \diff t \\
    &=
    +\infty.
\end{align}
What is new in Corollary \ref{Cor2} is the term
\begin{equation}
     \left(1-\frac{\|u\|_{\dot{H}^{\alpha-1}}^4}
    {\|u\|_{\dot{H}^{\alpha-2}}^2 \|u\|_{\dot{H}^\alpha}^2}
    \right)^\frac{p}{2},
\end{equation}
which measures the deficit in the interpolation inequality for the embedding $\dot{H}^{\alpha-2} \cap \dot{H}^{\alpha}
\hookrightarrow \dot{H}^{\alpha-1},$ stated below.

\begin{proposition} \label{HilbertInterpolateProp}
For all $\alpha>\frac{1}{2},$ we have the embedding
$\dot{H}^{\alpha-2}\cap \dot{H}^\alpha 
\hookrightarrow \dot{H}^{\alpha-1},$
with for all $f \in \dot{H}^{\alpha-2}\cap \dot{H}^\alpha,$
\begin{equation}
    \|f\|_{\dot{H}^{\alpha-1}}^2\leq
    \|f\|_{\dot{H}^{\alpha-2}}\|f\|_{\dot{H}^{\alpha}}
\end{equation}
\end{proposition}

\begin{proof}
Fix $f \in \dot{H}^{\alpha-2}\cap \dot{H}^\alpha.$
Using the fact that $(-\Delta)^{\frac{1}{2}}$ is self-adjoint, and applying H\"older's inequality we find that
\begin{align}
    \|f\|_{\dot{H}^{\alpha-1}}^2
    &=
    \left\|(-\Delta)^{\frac{\alpha}{2}-\frac{1}{2}}
    f\right\|_{L^2}^2 \\
    &=
    \left<(-\Delta)^{\frac{\alpha}{2}-1}f,
    (-\Delta)^{\frac{\alpha}{2}}f\right> \\
    &\leq \label{Holder}
    \left\|(-\Delta)^{\frac{\alpha}{2}-1}f\right\|_{L^2}
    \left\|(-\Delta)^{\frac{\alpha}{2}}f\right\|_{L^2}\\
    &=
    \|f\|_{\dot{H}^{\alpha-2}}\|f\|_{\dot{H}^{\alpha}}.
\end{align}
This completes the proof.
\end{proof}

We will note that this interpolation inequality is related to eigenfunctions of the Laplacian because the only inequality in this proof is H\"older's inequality in \eqref{Holder}, which holds with equality if
\begin{equation}
    (-\Delta)^{\frac{\alpha}{2}}f
    =\lambda (-\Delta)^{\frac{\alpha}{2}-1}f,
\end{equation}
which in turn would imply that
\begin{equation}
    -\Delta f=\lambda f,
\end{equation}
and therefore that $f$ is an eigenfunction of the Laplacian. This can happen on the torus when working with $f\in \dot{H}^{\alpha-1}\left(\mathbb{T}^3\right),$ for example the function $f(x)=\sin(2\pi x_1),$ 
is an eigenfunction of the Laplacian in
$f\in \dot{H}^{\alpha-1}\left(\mathbb{R}^3\right),$
with $-\Delta f= 4 \pi^2 f,$ but not on the whole space.
Nonetheless, the sharp constant in Proposition \ref{HilbertInterpolateProp} 
is $1$ for
$\dot{H}^{\alpha-2}\left(\mathbb{R}^3\right)
\cap \dot{H}^\alpha\left(\mathbb{R}^3\right),$
even though, because are no eigenfunctions of the Laplacian in
$\dot{H}^{\alpha-2}\left(\mathbb{R}^3\right)
\cap \dot{H}^\alpha\left(\mathbb{R}^3\right),$
this constant is not attained.
We will prove this by considering functions whose Fourier transforms are supported on an annulus in $\mathbb{R}^3.$ 

\begin{proposition} \label{NarrowBands}
Suppose $u\in \dot{H}^{\alpha-2} \cap \dot{H}^{\alpha},$ with $u$ not identically zero, and
\begin{equation}
    \supp{ \hat{u}} \subset \left\{ \xi\in\mathbb{R}^3:
    R_1 \leq |\xi| \leq R_2 \right\}.
\end{equation}
Then 
\begin{equation}
    \frac{\|u\|_{\dot{H}^{\alpha-1}}^2}
    {\|u\|_{\dot{H}^{\alpha-2}} \|u\|_{\dot{H}^\alpha}}
    \geq
    \frac{R_1^2}{R_2^2}.
\end{equation}
This condition can be stated equivalently as
\begin{equation}
   1-\frac{\|u\|_{\dot{H}^{\alpha-1}}^4}
    {\|u\|_{\dot{H}^{\alpha-2}}^2 \|u\|_{\dot{H}^\alpha}^2}
    \leq 1-\frac{R_1^4}{R_2^4}.
\end{equation}
\end{proposition}

\begin{proof}
First we will observe that 
\begin{equation} \label{SupportBound}
    \supp{ \hat{u}} \subset \left\{ \xi\in\mathbb{R}^3:
    R_1 \leq |\xi| \leq R_2 \right\}
\end{equation}
implies that
\begin{align}
\|u\|_{\dot{H}^{\alpha-1}}^2
&=
\int_{\mathbb{R}^3} 
\left(4 \pi^2 |\xi|^2\right)^{\alpha-1}
\left|\hat{u}(\xi)\right|^2 \diff\xi \\
&=
\int_{\mathbb{R}^3} 
4 \pi^2 |\xi|^2 \left(4 \pi^2 |\xi|^2\right)^{\alpha-2}
\left|\hat{u}(\xi)\right|^2 \diff\xi \\
&\geq
4 \pi^2 R_1^2 \int_{\mathbb{R}^3} 
\left(4 \pi^2 |\xi|^2\right)^{\alpha-2}
\left|\hat{u}(\xi)\right|^2 \diff\xi \\
&= \label{InnerBound}
4 \pi^2 R_1^2 \|u\|_{\dot{H}^{\alpha-2}}^2.
\end{align}

Likewise, the condition on the support of $\hat{u}$ \eqref{SupportBound} implies that
\begin{align}
    \|u\|_{\dot{H}^\alpha}^2
    &=
    \int_{\mathbb{R}^3} 
    \left(4 \pi^2 |\xi|^2\right)^{\alpha}
    \left|\hat{u}(\xi)\right|^2 \diff\xi \\
    &=
    \int_{\mathbb{R}^3}
    16 \pi^4|\xi|^4 \left(4 \pi^2 |\xi|^2\right)^{\alpha-2}
    \left|\hat{u}(\xi)\right|^2 \diff\xi \\
    &\leq
    16 \pi^4 R_2^4 \int_{\mathbb{R}^3}
    \left(4 \pi^2 |\xi|^2\right)^{\alpha-2}
    \left|\hat{u}(\xi)\right|^2 \diff\xi \\
    &= \label{OuterBound}
    16 \pi^4 R_2^4 \|u\|_{\dot{H}^{\alpha-2}}^2
\end{align}

Putting together \eqref{InnerBound} and \eqref{OuterBound}
we find that
\begin{align}
    \frac{\|u\|_{\dot{H}^{\alpha-1}}^4}
    {\|u\|_{\dot{H}^{\alpha-2}}^2\|u\|_{\dot{H}^{\alpha}}^2}
    &\geq
    \frac{16 \pi^4 R_1^4 \|u\|_{\dot{H}^{\alpha-2}}^4}
    {16 \pi^4 R_2^4 \|u\|_{\dot{H}^{\alpha-2}}^4}\\
    &=
    \frac{R_1^4}{R_2^4}.
\end{align}
It then immediately follows that
\begin{equation}
    \frac{\|u\|_{\dot{H}^{\alpha-1}}^2}
    {\|u\|_{\dot{H}^{\alpha-2}} \|u\|_{\dot{H}^\alpha}}
    \geq
    \frac{R_1^2}{R_2^2},
\end{equation}
and
\begin{equation}
    1-\frac{\|u\|_{\dot{H}^{\alpha-1}}^4}
    {\|u\|_{\dot{H}^{\alpha-2}}^2\|u\|_{\dot{H}^{\alpha}}^2}
    \leq
    1-\frac{R_1^4}{R_2^4}.
\end{equation}
This completes the proof.
\end{proof}
Proposition \ref{NarrowBands} shows that the sharp constant for the interpolation inequality in Proposition \ref{HilbertInterpolateProp} is in fact $1,$ and so the term
\begin{equation}
     \left(1-\frac{\|u\|_{\dot{H}^{\alpha-1}}^4}
    {\|u\|_{\dot{H}^{\alpha-2}}^2 \|u\|_{\dot{H}^\alpha}^2}
    \right)^\frac{p}{2}
\end{equation}
does indeed measure the deficit in this interpolation inequality.

Furthermore, Corollary \ref{Cor2} and Proposition \ref{NarrowBands} show that finite-time blowup solutions cannot concentrate on arbitrarily narrow bands in Fourier space, supported between an inner radius of $R_1(t)$ and an outer radius of $R_2(t),$ 
with the ratio $\frac{R_1(t)}{R_2(t)} \to 1$ 
arbitrarily quickly as $t \to T_{max}$ relative to the 
size of $\|u(\cdot,t)\|_{\dot{H}^\alpha}^p$.
In particular we will prove the following result, which is Corollary \ref{NarrowBandsRegCritIntro},
and is restated here for the reader's convenience.

\begin{corollary} \label{NarrowBandsRegCrit}
Suppose $u\in C\left(\left[0,T_{max}\right);H^1\right)$
is a mild solution of the Navier--Stokes equation,
and suppose for all $0<t<T_{max}$
\begin{equation}
    \supp{ \hat{u}(t)} \subset \left\{ \xi\in\mathbb{R}^3:
    R_1(t) \leq |\xi| \leq R_2(t) \right\}.
\end{equation}
Let $2\leq \alpha\leq \frac{5}{2}, 
\alpha=\frac{1}{2}+\frac{2}{p}.$ 
Then for all $0<t<T_{max}$
\begin{equation} 
    \|\nabla u(\cdot,t)\|_{L^2}^2\leq 
    \left\|\nabla u^0\right\|_{L^2}^2
    \exp\left(\Tilde{C}_\alpha \int_0^t
    \|u\|_{\dot{H}^\alpha}^p 
    \left(1-\frac{R_1(\tau)^4}{R_2(\tau))^4}
    \right)^\frac{p}{2} \diff\tau \right),
\end{equation}
where $\Tilde{C}_\alpha>0$ depends only on $\alpha.$
In particular, if $T_{max}<+\infty$ then
\begin{equation}
    \int_0^{T_{max}}
    \|u\|_{\dot{H}^\alpha}^p 
    \left(1-\frac{R_1(t)^4}{R_2(t))^4}
    \right)^\frac{p}{2} \diff t
    =+\infty.
\end{equation}
\end{corollary}

\begin{proof}
This corollary follows immediately from Corollary \ref{Cor2} and Proposition \ref{NarrowBands}.
\end{proof}

\begin{remark}
We will note that for a solution $u$ of the Navier--Stokes equation satisfying the hypotheses of Corollary \ref{NarrowBandsRegCrit},
if $T_{max}<+\infty,$ then clearly 
\begin{equation}
    \lim_{t\to T_{max}} R_2(t)=+\infty,
\end{equation}
otherwise we will have
\begin{equation}
    \liminf_{t\to T_{max}}
    \|\nabla u(\cdot,t)\|_{L^2}
    \leq 2\pi\left\|u^0\right\|_{L^2}
    \liminf_{t\to T_{max}}R_2(t)<+\infty,
\end{equation}
which contradicts the assumption that $T_{max}<+\infty$.
This means that Corollary \ref{NarrowBandsRegCrit} rules out concentration arbitrarily quickly on narrow bands in Fourier space in the sense that 
$\frac{R_1(t)}{R_2(t)}\to 1$, but these bands are not necessarily narrow in the sense that
$R_2(t)-R_1(t) \to 0$,
because $R_2(t) \to +\infty$.
\end{remark}

\bibliographystyle{amsplain}
\bibliography{Bib}

\begin{bibdiv}
\begin{biblist}

\bib{AlbrittonBesov}{article}{
      author={Albritton, Dallas},
       title={Blow-up criteria for the {N}avier-{S}tokes equations in
  non-endpoint critical {B}esov spaces},
        date={2018},
        ISSN={2157-5045},
     journal={Anal. PDE},
      volume={11},
      number={6},
       pages={1415\ndash 1456},
         url={https://doi.org/10.2140/apde.2018.11.1415},
      review={\MR{3803715}},
}

\bib{AlbrittonBarker}{article}{
      author={Albritton, Dallas},
      author={Barker, Tobias},
       title={Global weak {B}esov solutions of the {N}avier-{S}tokes equations
  and applications},
        date={2019},
        ISSN={0003-9527},
     journal={Arch. Ration. Mech. Anal.},
      volume={232},
      number={1},
       pages={197\ndash 263},
         url={https://doi.org/10.1007/s00205-018-1319-0},
      review={\MR{3916974}},
}

\bib{BealeKatoMajda}{article}{
      author={Beale, J.~T.},
      author={Kato, T.},
      author={Majda, A.},
       title={Remarks on the breakdown of smooth solutions for the {$3$}-{D}
  {E}uler equations},
        date={1984},
        ISSN={0010-3616},
     journal={Comm. Math. Phys.},
      volume={94},
      number={1},
       pages={61\ndash 66},
         url={http://projecteuclid.org/euclid.cmp/1103941230},
      review={\MR{763762}},
}

\bib{Bradshaw}{article}{
      author={Bradshaw, Z.},
      author={Gruji\'{c}, Z.},
       title={Frequency localized regularity criteria for the 3{D}
  {N}avier-{S}tokes equations},
        date={2017},
        ISSN={0003-9527},
     journal={Arch. Ration. Mech. Anal.},
      volume={224},
      number={1},
       pages={125\ndash 133},
         url={https://doi.org/10.1007/s00205-016-1069-9},
      review={\MR{3609247}},
}

\bib{ChaeVort}{article}{
      author={Chae, Dongho},
      author={Choe, Hi-Jun},
       title={Regularity of solutions to the {N}avier-{S}tokes equation},
        date={1999},
        ISSN={1072-6691},
     journal={Electron. J. Differential Equations},
       pages={No. 05, 7},
      review={\MR{1673067}},
}

\bib{Chemin1}{article}{
      author={Chemin, Jean-Yves},
      author={Zhang, Ping},
       title={On the critical one component regularity for 3-{D}
  {N}avier-{S}tokes systems},
        date={2016},
        ISSN={0012-9593},
     journal={Ann. Sci. \'{E}c. Norm. Sup\'{e}r. (4)},
      volume={49},
      number={1},
       pages={131\ndash 167},
         url={https://doi.org/10.24033/asens.2278},
      review={\MR{3465978}},
}

\bib{Chemin2}{article}{
      author={Chemin, Jean-Yves},
      author={Zhang, Ping},
      author={Zhang, Zhifei},
       title={On the critical one component regularity for 3-{D}
  {N}avier-{S}tokes system: general case},
        date={2017},
        ISSN={0003-9527},
     journal={Arch. Ration. Mech. Anal.},
      volume={224},
      number={3},
       pages={871\ndash 905},
         url={https://doi.org/10.1007/s00205-017-1089-0},
      review={\MR{3621812}},
}

\bib{ChenZhangBesov}{article}{
      author={Chen, Qionglei},
      author={Zhang, Zhifei},
       title={Space-time estimates in the {B}esov spaces and the
  {N}avier-{S}tokes equations},
        date={2006},
        ISSN={1073-2772},
     journal={Methods Appl. Anal.},
      volume={13},
      number={1},
       pages={107\ndash 122},
         url={https://doi.org/10.4310/MAA.2006.v13.n1.a6},
      review={\MR{2275874}},
}

\bib{Shvydkoy}{article}{
      author={Cheskidov, A.},
      author={Shvydkoy, R.},
       title={A unified approach to regularity problems for the 3{D}
  {N}avier-{S}tokes and {E}uler equations: the use of {K}olmogorov's
  dissipation range},
        date={2014},
        ISSN={1422-6928},
     journal={J. Math. Fluid Mech.},
      volume={16},
      number={2},
       pages={263\ndash 273},
         url={https://doi.org/10.1007/s00021-014-0167-4},
      review={\MR{3208714}},
}

\bib{SereginL3}{article}{
      author={Escauriaza, L.},
      author={Seregin, G.~A.},
      author={\v{S}ver\'ak, V.},
       title={{$L_{3,\infty}$}-solutions of {N}avier-{S}tokes equations and
  backward uniqueness},
        date={2003},
        ISSN={0042-1316},
     journal={Uspekhi Mat. Nauk},
      volume={58},
      number={2(350)},
       pages={3\ndash 44},
         url={https://doi.org/10.1070/RM2003v058n02ABEH000609},
      review={\MR{1992563}},
}

\bib{Kato}{article}{
      author={Fujita, Hiroshi},
      author={Kato, Tosio},
       title={On the {N}avier-{S}tokes initial value problem. {I}},
        date={1964},
        ISSN={0003-9527},
     journal={Arch. Rational Mech. Anal.},
      volume={16},
       pages={269\ndash 315},
         url={http://dx.doi.org/10.1007/BF00276188},
      review={\MR{0166499}},
}

\bib{GKPbesov}{article}{
      author={Gallagher, Isabelle},
      author={Koch, Gabriel~S.},
      author={Planchon, Fabrice},
       title={Blow-up of critical {B}esov norms at a potential
  {N}avier-{S}tokes singularity},
        date={2016},
        ISSN={0010-3616},
     journal={Comm. Math. Phys.},
      volume={343},
      number={1},
       pages={39\ndash 82},
         url={https://doi.org/10.1007/s00220-016-2593-z},
      review={\MR{3475661}},
}

\bib{Kolmogorov}{article}{
      author={Kolmogorov, A.},
       title={The local structure of turbulence in incompressible viscous fluid
  for very large {R}eynold's numbers},
        date={1941},
     journal={C. R. (Doklady) Acad. Sci. URSS (N.S.)},
      volume={30},
       pages={301\ndash 305},
      review={\MR{0004146}},
}

\bib{KOTbesov}{article}{
      author={Kozono, Hideo},
      author={Ogawa, Takayoshi},
      author={Taniuchi, Yasushi},
       title={The critical {S}obolev inequalities in {B}esov spaces and
  regularity criterion to some semi-linear evolution equations},
        date={2002},
        ISSN={0025-5874},
     journal={Math. Z.},
      volume={242},
      number={2},
       pages={251\ndash 278},
         url={https://doi.org/10.1007/s002090100332},
      review={\MR{1980623}},
}

\bib{KozonoShimadaBesov}{article}{
      author={Kozono, Hideo},
      author={Shimada, Yukihiro},
       title={Bilinear estimates in homogeneous {T}riebel-{L}izorkin spaces and
  the {N}avier-{S}tokes equations},
        date={2004},
        ISSN={0025-584X},
     journal={Math. Nachr.},
      volume={276},
       pages={63\ndash 74},
  url={https://doi-org.libaccess.lib.mcmaster.ca/10.1002/mana.200310213},
      review={\MR{2100048}},
}

\bib{Kukavica}{article}{
      author={Kukavica, Igor},
      author={Ziane, Mohammed},
       title={Navier-{S}tokes equations with regularity in one direction},
        date={2007},
        ISSN={0022-2488},
     journal={J. Math. Phys.},
      volume={48},
      number={6},
       pages={065203, 10},
         url={https://doi.org/10.1063/1.2395919},
      review={\MR{2337002}},
}

\bib{Ladyzhenskaya}{article}{
      author={Ladyzhenskaya, O.~A.},
       title={Uniqueness and smoothness of generalized solutions of
  {N}avier-{S}tokes equations},
        date={1967},
     journal={Zap. Nau\v cn. Sem. Leningrad. Otdel. Mat. Inst. Steklov.
  (LOMI)},
      volume={5},
       pages={169\ndash 185},
      review={\MR{0236541}},
}

\bib{Leray}{article}{
      author={Leray, Jean},
       title={Sur le mouvement d'un liquide visqueux emplissant l'espace},
        date={1934},
        ISSN={0001-5962},
     journal={Acta Math.},
      volume={63},
      number={1},
       pages={193\ndash 248},
         url={http://dx.doi.org/10.1007/BF02547354},
      review={\MR{1555394}},
}

\bib{Lieb}{article}{
      author={Lieb, Elliott~H.},
       title={Sharp constants in the {H}ardy-{L}ittlewood-{S}obolev and related
  inequalities},
        date={1983},
        ISSN={0003-486X},
     journal={Ann. of Math. (2)},
      volume={118},
      number={2},
       pages={349\ndash 374},
         url={https://doi.org/10.2307/2007032},
      review={\MR{717827}},
}

\bib{Luo}{article}{
      author={Luo, Xiaoyutao},
       title={A {B}eale-{K}ato-{M}ajda criterion with optimal frequency and
  temporal localization},
        date={2019},
        ISSN={1422-6928},
     journal={J. Math. Fluid Mech.},
      volume={21},
      number={1},
       pages={Paper No. 1, 16},
         url={https://doi.org/10.1007/s00021-019-0411-z},
      review={\MR{3902460}},
}

\bib{Obukhov}{article}{
      author={Obukhov, A.},
       title={On the energy distribution in the spectrum of a turbulent flow},
        date={1941},
     journal={C. R. (Doklady) Acad. Sci. URSS (N.S.)},
      volume={32},
       pages={19\ndash 21},
      review={\MR{0005852}},
}

\bib{Prodi}{article}{
      author={Prodi, Giovanni},
       title={Un teorema di unicit\`a per le equazioni di {N}avier-{S}tokes},
        date={1959},
        ISSN={0003-4622},
     journal={Ann. Mat. Pura Appl. (4)},
      volume={48},
       pages={173\ndash 182},
         url={https://doi.org/10.1007/BF02410664},
      review={\MR{0126088}},
}

\bib{Serrin}{article}{
      author={Serrin, James},
       title={On the interior regularity of weak solutions of the
  {N}avier-{S}tokes equations},
        date={1962},
        ISSN={0003-9527},
     journal={Arch. Rational Mech. Anal.},
      volume={9},
       pages={187\ndash 195},
         url={http://dx.doi.org/10.1007/BF00253344},
      review={\MR{0136885}},
}

\bib{Sobolev}{book}{
      author={Sobolev, S.~L.},
       title={Some applications of functional analysis in mathematical
  physics},
      series={Translations of Mathematical Monographs},
   publisher={American Mathematical Society, Providence, RI},
        date={1991},
      volume={90},
        ISBN={0-8218-4549-7},
        note={Translated from the third Russian edition by Harold H. McFaden,
  With comments by V. P. Palamodov},
      review={\MR{1125990}},
}

\bib{Talenti}{article}{
      author={Talenti, Giorgio},
       title={Best constant in {S}obolev inequality},
        date={1976},
        ISSN={0003-4622},
     journal={Ann. Mat. Pura Appl. (4)},
      volume={110},
       pages={353\ndash 372},
         url={https://doi.org/10.1007/BF02418013},
      review={\MR{0463908}},
}

\bib{TaoModel}{article}{
      author={Tao, Terence},
       title={Finite time blowup for an averaged three-dimensional
  {N}avier-{S}tokes equation},
        date={2016},
        ISSN={0894-0347},
     journal={J. Amer. Math. Soc.},
      volume={29},
      number={3},
       pages={601\ndash 674},
         url={https://doi.org/10.1090/jams/838},
      review={\MR{3486169}},
}

\end{biblist}
\end{bibdiv}
\end{document}